\documentclass[11pt]{article}

\usepackage[latin1]{inputenc}
\usepackage{epsfig}
\usepackage{graphicx,psfrag}
\usepackage[tight]{subfigure}

\usepackage{amsfonts,amsthm,bbm,amssymb,epsf,amsmath,%
latexsym,amscd,amsfonts,enumerate,amsthm,supertabular,color,url}

\newcommand{\RR}{\mathbb{R}}

\DeclareMathOperator{\rank}{rank}
\DeclareMathOperator{\diag}{diag}

\DeclareMathOperator{\spann}{span}

\newtheorem{theorem}{Theorem}

\newcommand{\ZZ}{\mathbb{Z}}

\def\hu{{\hat u}}

\def\bone{{\boldsymbol 1}}

\newcommand{\red}[1]{{\relax}}    %

\begin{document}

\title{Selection of Sparse Sets of Influence for Meshless Finite Difference Methods}%

\author{Oleg Davydov\thanks{Department of Mathematics, University of Giessen, Arndtstrasse 2, 
35392 Giessen, Germany, \tt{oleg.davydov@math.uni-giessen.de}}}%

\maketitle

\begin{abstract}
 We suggest an efficient algorithm for the selection of sparse subsets of a set of influence for the numerical discretization of differential
operators on irregular nodes with polynomial consistency of a given
order with the help of the QR decomposition of an appropriately weighted polynomial collocation matrix, and prove that the accuracy of the  resulting numerical
differentiation formulas is comparable with that of the formulas generated on the original set of influence.

\end{abstract}

\section{Introduction}\label{intro}
Meshless finite difference methods discretize a boundary value problem
\begin{align}\label{pr}
	\begin{split}
 	L u &= f\; \text{in}\; \Omega, \\
	Bu &= g \; \text{in}\; \partial\Omega,
	\end{split}
\end{align}
with the help of numerical differentiation formulas of the type
\begin{equation}\label{ndif} 
D u(z)\approx\sum_{j=1}^{n} w_j u( y_j ), \quad z,\,y_1,\ldots,y_n\in \RR^d,
\end{equation}
on an irregular set of nodes $Y=\{y_1,\ldots,y_n\}$,
where  $D$ is a linear differential operator 
\begin{equation}\label{Dop}  
Du=\sum_{\alpha\in\mathbb{Z}_+^d\atop|\alpha|\le \kappa}c_\alpha\partial^\alpha u,
\quad
\partial^\alpha:=\frac{\partial^{|\alpha|}}{\partial x^{\alpha}}=
\frac{\partial^{|\alpha|}}{\partial x_1^{\alpha_1}\cdots\partial x_d^{\alpha_d}},
\quad |\alpha|=\alpha_1+\cdots+\alpha_d,
\end{equation}
with variable coefficients $c_\alpha$ and order $\kappa=\kappa(D)$. Usually, one or more such operators are associated with a given problem \eqref{pr}  
by linearizing $L$ and $B$ and extracting their parts of different character, such as 
the diffusion or convection term, see for example
Both $z$ and and its  \emph{set of influence} $Y$ belong to a finite set $X=\{x_1,\ldots,x_N\}$ of (unconnected) nodes
that discretize the whole domain $\overline\Omega$, and a discrete solution $\hu\in\RR^N$ is sought 
as an approximation to $u|_X$.

For example, the Dirichlet problem for the Poisson equation ($L=\Delta$ and $B=I$) can be discretized by numerical
differentiation of the Laplacian 
$$
\Delta u(x_i)\approx\sum_{j\in J_i} w_{ij}\, u( x_j ),\quad J_i\subset \{1,\ldots,N\},$$
and $\hu$ obtained by solving the linear system
\begin{align*}
\sum_{j\in J_i} w_{ij}\, \hu_j = f(x_i),\quad x_i\in\Omega;\qquad
\hu_i = g(x_i),\quad x_i\in\partial\Omega,
\end{align*}
where $w_{ij}:=0$ whenever $j\notin J_i$.

Similar to the classical finite difference method, the error of the formula \eqref{ndif} plays the role of the consistency or
the local discretization error. This error may be reduced by choosing larger sets of influence $X_i=\{x_j:j\in J_i\}$ and a higher order
numerical differentiation method, giving rise to a higher convergence order of the numerical solution $\hu$. 
 It is however important to avoid unnecessarily large sets of influence that do not significantly reduce the 
consistency error of \eqref{ndif}.
Smaller sets of influence lead to sparser linear systems to be solved for
$\hat u$. For example, sets of influence consisting of just 7 points are generated by the algorithms suggested in 
\cite{DavyOanh11,OanhDavyPhu17} for elliptic problems, which helps to produce adaptive methless methods that compete with the
piecewise linear finite elements in terms of both accuracy and sparsity of the system matrix. 
However, the algorithms of \cite{DavyOanh11,OanhDavyPhu17}  are geometric in nature and therefore seem difficult to extend to higher
order methods.

Most work on meshless finite difference methods relies on selecting the sets of influence in a
very simple way by forming $X_i$ from an \emph{ad hoc} number of nearest neighbors of $x_i$, see e.g.~\cite{BFFB17}.  
This approach works well and produces relatively small sets of influence when the global node set
$X$ is carefully generated
(\emph{node generation}: the counterpart of mesh generation in mesh based methods). Several node generation methods 
have been developed.
On the other hand, one of the main goals of meshless methods is to avoid sophisticated mesh generation. It is therefore
desirable to develop approaches that let meshless finite differences perform well also on nodes generated by simple
methods allowing local irregularities that would lead to a severe diteriation in the performance of mesh-based methods.

When node generation is inexpensive and the set $X$ is suboptimal, then it is usually possible to obtain acceptable consistency error by selecting larger sets
of influence in the locations affected by irregularities. However, if the sets of influence are controlled by a single
number of nearest neighbors, then the method  unneccesarily uses too many nodes in locations where the neighborhood is more
regular than in the worst locations, which leads to unneccesary increased density of the system matrix $[w_{ij}]_{ij}$. 
In this case the number of nearest neighbors that guarantee good numerical differentiation error 
may be too high, and
selection of sutable small subsets particularly important.

In this paper we discuss how to reduce the size of a set of influence while keeping essentially the 
same consistency error achieved on the original set. In particular, we suggest a new efficient method for the calculation of
sparse weights based on pivoted QR factorization of the polynomial collocation matrices.

\section{Consistency error estimates}

It has been shown in \cite{DavySchaback16,DavySchaback18} that the error of the kernel-based formulas \eqref{ndif}  as well as 
certain (\emph{minimal}) polynomial type formulas can be bounded by the \emph{growth function} 
$$
\rho_{q,D}(z,Y)=\sup\big\{D p(z):p\in\Pi^d_q,\;|p(y_j)|\le\|y_j-z\|_2^{q},\;
j=1,\ldots,n\big\}$$
times a factor depending on the smoothness of $f$ and independent of the geometry of the set of influence 
$Y=\{y_1,\ldots,y_n\}$.
Here $\Pi^d_{q}$ denotes the space of all $d$-variate polynomials of order 
at most $q$, i.e.\ of total degree less than $q$, with $\Pi^d_{0}:=\{0\}$. 

In particular, in the polynomial case a duality theorem shows that
$\rho_{q,D}(z,Y)$ is the minimum of
\begin{equation}\label{ell1}
\|w\|_{1,q}:=\sum_{j=1}^{n} |w_j|\|y_j-z\|_2^{q}
\end{equation}
subject to the exactness condition
\begin{equation}\label{exact}
D p(z)=\sum_{j=1}^{n} w_j p( y_j ) \quad\hbox{for all}\quad  p\in \Pi^d_{q}.
\end{equation}
On the other hand, the following error bound holds for any formula \eqref{ndif} satisfying \eqref{exact} 
with $q>\kappa(D)$, and all $f\in C^{q-1}(\Omega)$ with Lipschitz continuous derivatives of order $q$,
\begin{equation}\label{error}
|D f(z)-\sum_{j=1}^{n} w_j f( y_j )|\le \sum_{j=1}^{n} |w_j|\|y_j-z\|_2^{q}\,|f|_{q,\Omega},
\end{equation}
where $\Omega\subset\RR^d$ is any domain that 
contains the set   
\begin{equation}\label{SzX}
S_{z,Y}:=\bigcup_{i=1}^n[z,y_i],\qquad
[x,y]:=\{\alpha x+(1-\alpha)y: 0\le\alpha\le 1\},
\end{equation}
and 
\begin{equation}\label{Hoe}
|f|_{q,\Omega}:=\frac{1}{q!}
\Big(\sum_{|\alpha|=q-1}{\textstyle{q-1\choose \alpha}}|\partial^\alpha f|_{1,\Omega}^2\Big)^{1/2},
\quad |f|_{1,\Omega}:=\sup_{x,y\in\Omega\atop x\ne y}\frac{|f(x)-f(y)|}{\|x-y\|_2}.
\end{equation}

It follows that for the \emph{$\ell_1$-minimal} formula of order $q>\kappa(D)$ whose weight vector 
$w^*=[w^*_1,\ldots,w^*_n]^T$ is computed by minimizing
$\|w\|_{1,q}$ subject to \eqref{exact},
\begin{equation}\label{ell1er}
|D f(z)-\sum_{j=1}^{n} w^*_j f( y_j )|\le\rho_{q,D}(z,Y)|f|_{q,\Omega},
\end{equation}
which is the best bound obtainable from \eqref{error}.

For the \emph{$\ell_2$-minimal} formula with weights $w^{**}_j$, $j=1,\ldots,n$, 
obtained by minimizing
\begin{equation}\label{ell2}
\|w\|_{2,q}^2:=\sum_{j=1}^{n} w_j^2\|y_j-z\|_2^{2q}
\end{equation}
subject to \eqref{exact}, the error bound of \cite{DavySchaback18} is  worse only by the factor $\sqrt{n}$,
\begin{align}
\label{ell2er}
|D f(z)-\sum_{j=1}^{n} w^{**}_j f( y_j )|&\le\sqrt{n}\,\rho_{q,D}(z,Y)|f|_{q,\Omega}.
\end{align}

Moreover, the growth function concides with $\|w^*\|_{1,q}$, as mentioned before, and 
it can be estimated with the help of $\|w^{**}\|_{2,q}$:
\begin{equation}\label{rhoest}
\rho_{q,D}(z,Y)=\|w^*\|_{1,q},\qquad 
\|w^{**}\|_{2,q}\le\rho_{q,D}(z,Y)\le \sqrt{n}\,\|w^{**}\|_{2,q}.
\end{equation}

Note that $\ell_2$-minimal formulas can be interpreted as obtained by differentiating a least squares polynomial 
of order $q$ to
the data at $Y$, with weights  $\|y_j-z\|_2^{-2q}$, see \cite[Section 5]{DavySchaback18}, which has been frequently used in meshless finite 
difference methods, albeit usually with different weights that do not satisfy the error bound
\eqref{ell2er}. The $\ell_1$-minimal formulas have been considered in 
\cite{Seibold08,BaMoKi2011} for the Laplacian operator $D=\Delta$,  with an additional requirement of positivity that 
ensures the $L$-matrix property of the system matrix for the Poisson problem, but may only
be satisfied for $q\le4$, see also \cite[Section 4]{DavySchaback18}.

\section{Sparse subsets of sets of influence}
Since the growth function is monotone decreasing with respect to $Y$, that is
\begin{equation}\label{monot}
\rho_{q,D}(z,Y'')\le \rho_{q,D}(z,Y')\quad \text{if}\quad Y'\subset Y'',
\end{equation}
the estimates \eqref{ell1er} and \eqref{ell2er} and their kernel-based counterparts in \cite{DavySchaback16} generally 
improve when larger sets of neighbors are used.

A simple way to produce a set of influence $Y=X_i$ is by selecting 
a certain number $m$ of nearest neighbors of $z=x_i$ in $X$, where 
the size $m$ is a sufficiently large number choosen on the basis of experience  depending on the number of variables
$d$ and expected convergence order. For numerical differentiation that involves polynomials or order $q$, 
$m$ is typically choosen to be at least the double of 
$$
\nu_{q,d}:={q-1+d\choose d}=\dim \Pi^d_q,$$ 
since relying on a number  of  nearest neighbors less than or only slightly exceeding $\nu_{q,d}$
risks low consistency order or numerical instability even for geometrically 
nicely distributed node sets. If node generation is performed by less sophisticated algorithms, then the number of
nearest neighbors needed to guarantee a good consistency error may even be significanly higher than $2\nu_{q,d}$.

Since the density of the system matrix is determined by the sizes of the sets of influence, it is natural to try to reduce
these sizes whenever possible if this does not cause a significant increase of the consistency error. 
In view of the role of the growth function as an indicator of the consistency error obtainable on a given infuence set, we
consider the problem of finding a significantly smaller subset $\tilde Y$ of a given set of influence 
$Y=\{y_1,\ldots,y_m\}$ such that
\begin{equation}\label{rhob}
\rho_{q,D}(z,\tilde Y)\le C\rho_{q,D}(z,Y)
\end{equation}
for some small constant $C\ge 1$.

In fact, $\ell_1$-minimal formulas already generate a subset $Y^*\subset Y$ of size $|Y^*|\le\nu_{q,d}$ with
$\rho_{q,D}(z,Y^*)=\rho_{q,D}(z,Y)$ as soon as $\rho_{q,D}(z,Y)<\infty$. Indeed, many weights $w_j$ vanish when
\eqref{ell1} is minimized suject to \eqref{exact} because this optimization problem can be interpeted as a linear program,
see \cite[Section 4]{DavySchaback18} for more details. The papers \cite{Seibold08,BaMoKi2011} have applied so obtained 
sparse positive formulas and subsets $\tilde Y$ in the meshless finite difference method. 

Moreover, $\ell_1$-minimal formulas are often of even smaller size than $\nu_{q,d}$ if the set $Y$ allows this, for example if
$d=2$, $D=\Delta$, $q=3$ or 4, $z\in Y$, and $Y$ contains a sufficiently localized 5-star subset centered at $z$. In this case
classical 5-point stencil of the finite difference method is automatically recovered. 

Unfortunately, $\ell_1$-minimal formulas are relatively expensive to compute and numerical methods for them are not always
reliable. Therefore we suggest an alternative, significantly more efficient method of
selecting a sparse subset $\tilde Y\subset Y$ satisfying \eqref{rhob} with a constant $C$ estimated in Theorem~\ref{pqrbound}. 

We assume without loss of generality that 
$$
z=0 \quad\text{and}\quad z\notin\{y_2,\ldots,y_m\},$$
 which still allows $z=y_1$. Let $p_1,\ldots,p_\nu$ be a basis for $\Pi^d_{q}$,
$$
\Pi^d_{q}=\spann\{p_1,\ldots,p_\nu\},\quad \nu=\nu_{q,d},$$
with 
\begin{equation}\label{p1cond}
p_1\equiv1\text{ and } p_2(z)=\cdots=p_\nu(z)=0\quad\text{if}\quad z=y_1.
\end{equation}
In particular, 
after an appropriate translation and scaling of the coordinate system of $\RR^d$ the monomial basis
$$
y^\alpha,\quad \alpha\in\ZZ^d_+,\quad |\alpha|:=\sum_{i=1}^d \alpha_i< q,$$
may be used since the growth function is scale invariant \cite[Section 2]{DavySchaback18}, see also a discussion of 
the scalability of numerical diffrentiation formulas in 
\cite{DavySchaback18,DavySchaback19}. 

The exactness condition \eqref{exact} is equivalent to the system of linear equations
\begin{equation}\label{Ab}
Aw=b,\quad\text{with}\quad
A:=[p_i(y_j)]_{i,j=1}^{\nu,m}\in\RR^{\nu\times m},\quad b:=[Dp_i(z)]_{i=1}^m\in\RR^m,
\end{equation}
which is consistent if and only if $\rho_{q,D}(z,Y)<\infty$ \cite[Theorem 9]{DavySchaback18}. 
Typically (but not necessarily) $m\ge\nu$. 
An efficient method for computing a sparse solution of a consistent linear system is to employ the QR factorization 
of $A$ with column pivoting, see  \cite[Section 12.2.1]{GoVanL96}. It produces $w$ with at most $\rank(A)$ nonzero
components and has been successfully applied to the multivariate Vandermonde matrices 
in order to select good points for polynomial interpolation on domains \cite{SoVia09}. Applied to $A$ directly, 
this method however does not seem to produce useful sets of influence for mesless finite difference methods.

We suggest to apply a pivoted QR factorization after rescaling the system \eqref{Ab} 
with the help of the diagonal matrix
$$
\Theta:=\diag(\theta_1,\ldots,\theta_m),\qquad \theta_j=\|y_j-z\|_2^{-q},\quad j=1,\ldots,m.$$

We assume that $\rho_{q,D}(z,Y)<\infty$ such that \eqref{Ab} has at least one solution.
If $z\ne y_1$, then we transform the linear system $Aw=b$ in the form
$$
\tilde A v=b,\qquad\tilde A=A\Theta,\quad w=\Theta v,$$
and compute a QR factorization of $\tilde A$ with column pivoting,
$$ %
\tilde AP=Q
\begin{bmatrix} 
A_1 & A_2\\ 
0 & 0
\end{bmatrix},
$$ %
where $P$ is a permutation matrix, $Q$ an orthogonal matrix and $A_1\in\RR^{r\times r}$ is upper triangular and nonsingular,
with $r:=\rank(\tilde A)=\rank(A)$ \cite[Section 5.4.1]{GoVanL96}. Let $s$ be the largest index of the nonzero components 
of $Q^Tb$, such that
\begin{equation}\label{ndef}
Q^Tb=
\begin{bmatrix} 
\tilde b \\ 
0 
\end{bmatrix}
\qquad\text{with}\quad \tilde b\in\RR^s.
\end{equation}
Since \eqref{Ab} is consistent, the last $\nu-r$ components
of $Q^Tb$ must be zero, hence $s\le r$. We rewrite $\tilde AP$ in the form
\begin{equation}\label{pqr}
\tilde AP=Q
\begin{bmatrix} 
R_1 & R_2\\ 
0 & T
\end{bmatrix},\qquad R_1\in \RR^{s\times s},
\end{equation}
with an upper triangular and nonsingular matrix $R_1$.
Then the equation $\tilde A v=b$ is equivalent to
\begin{equation}\label{pqre}
\begin{bmatrix} 
R_1 & R_2\\ 
0 & T
\end{bmatrix}
\tilde v=
\begin{bmatrix} 
\tilde b \\ 
0 
\end{bmatrix},\qquad 
v=P\tilde v,\quad
\end{equation}
which has a sparse solution $\tilde v^\circ$ with at most $s$ nonzero components determined by the conditions
\begin{equation}\label{pqrs}
R_1[\tilde v^\circ_i]_{i=1}^s = \tilde b,\quad [\tilde v^\circ_i]_{i=s+1}^m=0.
\end{equation}
Then the vector
$$
w^\circ:=\Theta P\tilde v^\circ$$
also has at most $s\le \rank(A)$ nonzero components and satisfies \eqref{Ab}. We denote by $Y^\circ$ the subset of $Y$ 
that corresponds to the nonzero components of the vector $w^\circ$.

In the case $z=y_1$ we have assumed \eqref{p1cond}, in particular  $p_1\equiv 1$. Hence by \eqref{Dop} $Dp_1(z)=c_0(z)$. We 
replace $Aw=b$ by the equivalent equations
\begin{align*}
\sum_{j=1}^mw_j&=c_0(z),\\
\tilde A v&=b',\qquad 
w=\begin{bmatrix} 
w_1 \\ 
\Theta v
\end{bmatrix},
\end{align*}
where in this case
$$
\tilde A:=A'\Theta,\qquad A':=[p_i(y_j)]_{i,j=2}^{\nu,m},\quad 
b=\begin{bmatrix} 
c_0(z) \\ 
b' 
\end{bmatrix}$$
and
$$
\Theta:=\diag(\theta_2,\ldots,\theta_m),\qquad \theta_j=\|y_j-z\|_2^{-q},\quad j=2,\ldots,m.$$
Note that $\rank(\tilde A)=\rank(A')=\rank(A)-1$ thanks to \eqref{p1cond}.
After computing a pivoted QR factorization of $\tilde A$ in the form \eqref{pqr}, and 
a solution $\tilde v^\circ$ of \eqref{pqre} satisfying \eqref{pqrs}, we obtain 
$\tilde w=[\tilde w_j]_{j=2}^m:=\Theta P\tilde v^\circ\in\RR^{m-1}$, and
the vector
$$
w^\circ:=\begin{bmatrix} 
c_0(z)-\sum_{j=2}^{m}\tilde w_j \\ 
\tilde w 
\end{bmatrix}$$
satisfies \eqref{Ab} and has at most $s+1\le \rank(\tilde A)+1=\rank(A)$ nonzero components,
where $s$ is the largest index of nonzero components of $Q^Tb'$. 
The subset of $Y$ 
that corresponds to the nonzero components of the vector $w^\circ$ is again denoted $Y^\circ$. Note that in this case
$z=y_1\in Y^\circ$.

The rank of the matrix $A$ is independent of the choice of the basis $p_1,\ldots,p_\nu$ of $\Pi^d_q$, and we denote it
$r_q(Y).$ The following theorem provides a bound  of the type \eqref{rhob} for the subset $Y^\circ$.

\begin{theorem}\label{pqrbound} Assume that $\rho_{q,D}(z,Y)<\infty$. Let $Y^\circ\subset Y$ be constructed with the help of
a pivoted QR factorization as described above, and let $w^\circ$ denote the corresponding weght vector. Then
\begin{align}
\label{pqrb1}
\|w^\circ\|_{2,q}&\le (1+\|R_1^{-1}R_2\|_2^2)^{1/2}\|w^{**}\|_{2,q},\\
\label{pqrb2}
\rho_{q,D}(z,Y^\circ)&\le n^{1/2}(1+\|R_1^{-1}R_2\|_2^2)^{1/2}\rho_{q,D}(z,Y),
\end{align}
where the matrices $R_1,R_2$ are defined by \eqref{pqr}, and $n=|Y^\circ|\le r_q(Y)$ is the number of 
nonzero components of $w^\circ$.
\end{theorem}

\begin{proof}
We first consider the case when $z\ne y_1$. Let $w^{**}$ be the $\ell_2$-minimal weight vector that minimizes \eqref{ell2}
subject to \eqref{Ab}. Then $v^{**}=\Theta^{-1}w^{**}$ is the minimal 2-norm solution of $\tilde Av=b$ since 
$\|v^{**}\|_2=\|w^{**}\|_{2,q}$, and $\tilde v^{**}=P^{-1}v^{**}$ is the minimal 2-norm solution of
\eqref{pqre}. We write any solution $\tilde v$ of \eqref{pqre} in the form
$$ 
\tilde v=
\begin{bmatrix} 
\tilde v_1 \\ 
\tilde v_2 
\end{bmatrix},\quad \tilde v'\in\RR^{s\times s},\qquad
\text{in particular}\quad
\tilde v^{**}=\begin{bmatrix} 
\tilde v^{**}_1 \\ 
\tilde v^{**}_2 
\end{bmatrix},\quad
\tilde v^\circ=\begin{bmatrix} 
\tilde v^\circ_1 \\ 
0 
\end{bmatrix}.$$
The condition $T \tilde v_2=0$ is equivalent to $\tilde v_2= S\hat v$ for a suitable $\hat v\in\RR^\ell$, where the columns of $S$ form an
orthonormal basis for the null space $N(T)$ of $T$, and $\ell$ is the dimension of $N(T)$. In particular,
$\tilde v^{**}_2= S\hat v^{**}$. Hence \eqref{pqre} is equivalent to
$$
\begin{bmatrix} 
R_1 & R_2S
\end{bmatrix}
\begin{bmatrix} 
\tilde v_1 \\ 
\hat v 
\end{bmatrix}= 
\tilde b. 
$$
Since $\|S\hat v\|_2=\|\hat v\|_2$, we have $\|\tilde v^{**}\|_2^2=\|\tilde v^{**}_1\|_2^2+\|\hat v^{**}\|_2^2$.
Therefore $\begin{bmatrix} 
\tilde v^{**}_1 \\ 
\hat v^{**} 
\end{bmatrix}$ 
is the minimal 2-norm solution of the equation in the last display, and by \cite[Section 5.5.6]{GoVanL96}, we obtain
$$ 
\|\tilde v^\circ\|_2\le(1+\|R_1^{-1}R_2S\|_2^2)^{1/2}\|\tilde v^{**}\|_2.$$
Since $\|\tilde v^\circ\|_2=\|w^\circ\|_{2,q}$, $\|\tilde v^{**}\|_2=\|w^{**}\|_{2,q}$ and
$\|S\|_2= 1$, \eqref{pqrb1} follows.
Since the components of $w^\circ$ corresponding to $y_j\in Y^\circ$ 
form the only solution of $A^\circ w=b$, where
$$
A^\circ:=[p_i(y_j)]_{i=1,\,y_j\in Y^\circ}^{\nu},$$%
it follows by \eqref{rhoest} that
$$
\rho_{q,D}(z,Y^\circ)\le n^{1/2}\|w^\circ\|_{2,q}\quad\text{and}\quad
\|w^{**}\|_{2,q}\le \rho_{q,D}(z,Y),$$
and  \eqref{pqrb2} follows from \eqref{pqrb1}.

Let now $z= y_1$. By the same arguments we obtain the estimate
$$
\|\tilde w\|_{2,q}=\Big(\sum_{j=2}^{m} \tilde w_j^2\|y_j-z\|_2^{2q}\Big)^{1/2}
\le (1+\|R_1^{-1}R_2\|_2^2)^{1/2}\|\tilde w^{**}\|_{2,q},$$
where $\tilde w^{**}=[\tilde w_j^{**}]_{j=2}^m:=\Theta P\tilde v^{**}$ for the minimal 2-norm solution $\tilde v^{**}$ of $\tilde Av=b'$.
Since $\|y_1-z\|_2=0$, we have $\|w^\circ\|_{2,q}=\|\tilde w\|_{2,q}$.
Moreover, in view of \eqref{p1cond},
$$
A=\begin{bmatrix} 
1 & \bone^T\\ 
0 & A'
\end{bmatrix},\qquad \bone:=[1\cdots 1]^T\in\RR^{m-1}.$$
Hence, \eqref{Ab} is equivalent to
$$
\sum_{j=1}^mw_j=c_0(z),\quad A'[w_j]_{j=2}^m=b',$$
which implies that the $\ell_2$-minimal weight vector $w^{**}$ that minimizes 
$$
\sum_{j=1}^{m} w_j^2\|y_j-z\|_2^{2q}=\sum_{j=2}^{m} w_j^2\|y_j-z\|_2^{2q}$$
subject to \eqref{Ab} is given by
$$
w^{**}=\begin{bmatrix} 
c_0(z)-\sum_{j=2}^{m}\tilde w_j^{**} \\ 
\tilde w^{**} 
\end{bmatrix},$$
since $\tilde w^{**}$ %
minimizes 
$$
\sum_{j=2}^{m} w_j^2\|y_j-z\|_2^{2q}\qquad\text{subject to}\quad A'[w_j]_{j=2}^m=b'.$$
This implies $\|w^{**}\|_{2,q}=\|\tilde w^{**}\|_{2,q}$, and \eqref{pqrb1} folows.
The bound \eqref{pqrb2} is inferred from \eqref{pqrb1} by the same argument as before.
\end{proof}

\bibliographystyle{abbrv} %
\bibliography{meshless}

\begin{thebibliography}{10}

\bibitem{BFFB17}
V.~Bayona, N.~Flyer, B.~Fornberg, and G.~A. Barnett.
\newblock On the role of polynomials in {RBF-FD} approximations: {II}.
  {N}umerical solution of elliptic {PDE}s.
\newblock {\em Journal of Computational Physics}, 332:257 -- 273, 2017.

\bibitem{BaMoKi2011}
V.~Bayona, M.~Moscoso, and M.~Kindelan.
\newblock Optimal constant shape parameter for multiquadric based {RBF-FD}
  method.
\newblock {\em J. Comput. Phys.}, 230(19):7384--7399, 2011.

\bibitem{DavyOanh11}
O.~Davydov and D.~T. Oanh.
\newblock Adaptive meshless centres and {RBF} stencils for {P}oisson equation.
\newblock {\em J. Comput. Phys.}, 230:287--304, 2011.

\bibitem{DavySchaback16}
O.~Davydov and R.~Schaback.
\newblock Error bounds for kernel-based numerical differentiation.
\newblock {\em Numerische Mathematik}, 132(2):243--269, 2016.

\bibitem{DavySchaback18}
O.~Davydov and R.~Schaback.
\newblock Minimal numerical differentiation formulas.
\newblock {\em Numerische Mathematik}, 140(3):555--592, 2018.

\bibitem{DavySchaback19}
O.~Davydov and R.~Schaback.
\newblock {Optimal stencils in {S}obolev spaces}.
\newblock {\em IMA Journal of Numerical Analysis}, 39(1):398--422, 2019.

\bibitem{GoVanL96}
G.~H. Golub and C.~F. Van~Loan.
\newblock {\em Matrix Computations}.
\newblock The Johns Hopkins University Press, third edition, 1996.

\bibitem{GuEis96}
M.~Gu and S.~Eisenstat.
\newblock Efficient algorithms for computing a strong rank-revealing {QR}
  factorization.
\newblock {\em SIAM Journal on Scientific Computing}, 17(4):848--869, 1996.

\bibitem{OanhDavyPhu17}
D.~T. Oanh, O.~Davydov, and H.~X. Phu.
\newblock Adaptive {RBF-FD} method for elliptic problems with point
  singularities in 2{D}.
\newblock {\em Applied Mathematics and Computation}, 313:474--497, 2017.

\bibitem{Seibold08}
B.~Seibold.
\newblock Minimal positive stencils in meshfree finite difference methods for
  the {P}oisson equation.
\newblock {\em Comput. Methods Appl. Mech. Eng.}, 198(3-4):592--601, 2008.

\bibitem{SoVia09}
A.~Sommariva and M.~Vianello.
\newblock Computing approximate {F}ekete points by {QR} factorizations of
  {V}andermonde matrices.
\newblock {\em Computers \& Mathematics with Applications}, 57(8):1324 -- 1336,
  2009.

\end{thebibliography}

\end{document}